\documentclass[10pt]{elsarticle}
\usepackage[cp1251]{inputenc}
\usepackage[english]{babel}
\usepackage{amsmath}
\usepackage{amssymb}
\usepackage{amsfonts}
\usepackage{mathtools}
\usepackage{dsfont}
\usepackage{rotating}
\usepackage{graphicx}
\usepackage{floatflt,epsfig}
\usepackage{lineno,hyperref}
\usepackage{enumerate}
\usepackage{colortbl}
\usepackage{array,tabularx,tabulary,booktabs}
\usepackage{longtable}
\usepackage{multirow}
\usepackage{wrapfig}
\usepackage{subcaption}
\usepackage[table]{xcolor}
\newcolumntype{^}{>{\currentrowstyle}}

\journal{Arxiv}
\newtheorem{lemma}{Lemma}

\newtheorem{theorem}{Theorem}
\newtheorem{corollary}{Corollary}
\newtheorem{proposition}{Proposition}
\newtheorem{problem}{Problem}
\newtheorem{construction}{Construction}

\DeclareMathOperator{\Ker}{Ker}
\DeclareMathOperator{\image}{Im}
\DeclareMathOperator{\Aut}{Aut}
\DeclareMathOperator{\AG}{AG}

\begin{document}
\renewcommand{\abstractname}{Abstract}
\renewcommand{\refname}{References}
\renewcommand{\tablename}{Figure.}
\renewcommand{\arraystretch}{0.9}
\thispagestyle{empty}
\sloppy

\begin{frontmatter}
\title{On a correspondence between maximal cliques in Paley graphs of square order}

\author[01]{Sergey Goryainov}
\ead{sergey.goryainov3@gmail.com}

\author[02]{Alexander Masley}
\ead{masley.alexander@gmail.com }

\author[03]{Leonid Shalaginov}
\ead{44sh@mail.ru}

\address[01] {School of Mathematical Sciences,\\Hebei International Joint Research Center\\ for Mathematics and Interdisciplinary Science,\\ Hebei Normal University, Shijiazhuang  050024, P.R. China}
\address[02]{Technion -- Israel Institute of Technology,
CS Taub Building,\\ Haifa 3200003, Israel}
\address[03] {Chelyabinsk State University, 129 Bratiev Kashirinykh st.,\\Chelyabinsk 454021, Russia}


\begin{abstract}
Let $q$ be an odd prime power. Denote by $r(q)$ the value of $q$ modulo 4.
In this paper, we establish a linear fractional correspondence between two types of maximal cliques of size $\frac{q+r(q)}{2}$ in the Paley graph of order $q^2$.
\end{abstract}

\begin{keyword}
Paley graph; maximal clique;
linear fractional transformation
\vspace{\baselineskip}
\MSC[2010] 05C25\sep 11T30\sep 05E30\sep 30C10\sep 51E15
\end{keyword}
\end{frontmatter}

\section{Introduction}

Let $q_1$ be an odd prime power such that $q_1 \equiv 1 (4)$.
The \emph{Paley graph} $P(q_1)$ is the graph whose vertices are the elements of the finite field $\mathbb{F}_{q_1}$ with two vertices being adjacent
whenever their difference is a square in the multiplicative group~$\mathbb{F}_{q_1}^*$.

In this paper, we study maximal cliques in Paley graphs of square order $P(q^2)$, where $q$ is an odd prime power.
A Paley graph $P(q^2)$ is known to be a self-complementary strongly regular graph with parameters $\left( q^2, \frac{q^2-1}{2}, \frac{q^2-5}{4}, \frac{q^2-1}{4} \right)$ and  smallest eigenvalue $\frac{-1-q}{2}$.
For a strongly regular graph with valency $k$ and smallest eigenvalue $-m$, Delsarte proved (see \cite[Section 3.3.2]{D73}) that the size of a clique is at most $1 + \frac{k}{m}$.
It follows from the Delsarte bound that a maximum clique (as well as a maximum independent set) in $P(q^2)$ has size at most $q$.
The subfield $\mathbb{F}_q$ is an example of a clique of size $q$ in $P(q^2)$, which means that the Delsarte bound is tight for Paley graphs of square order.

For any odd prime power $q$, the field $\mathbb{F}_{q^2}$ can be naturally viewed as the affine plane $\AG(2,q)$.
From this point of view, all lines in $\AG(2,q)$ can be divided into two classes: quadratic lines and non-quadratic lines.
For a quadratic (resp. non-quadratic) line, the difference between any two distinct elements from this line is a square (resp. a non-square).
Thus the parallel classes of quadratic (resp. non-quadratic) lines give examples of partitions of vertices of $P(q^2)$ into maximum cliques (resp. maximum independent sets).
In particular, the subfield $\mathbb{F}_q$ in $\mathbb{F}_{q^2}$, being a quadratic line, is a clique of size $q$ in $P(q^2)$, which defines a parallel class consisting of quadratic lines in $\AG(2,q)$.
In \cite{B84}, Blokhuis showed that a clique (resp. an independent set) of size $q$ in $P(q^2)$ is necessarily a quadratic (resp. non-quadratic) line.

The following problem was first studied in \cite{BEHW96}.
    \begin{problem}\label{Prob1}
    What are maximal but not maximum cliques in Paley graphs of square order?
    \end{problem}
Since the definition of Paley graphs naturally combines addition and multiplication operations, the investigation of Problem \ref{Prob1} contributes to the structure theory of finite fields.

In the proof of \cite[Theorem 5.2]{HS91}, it was implicitly shown that, for any non-negative integer $m$ and relatively large $q$, there is a maximal clique of size approximately $\frac{q}{2^m}$ in $P(q^2)$. In particular, the case $m=0$ corresponds to the maximum clique $\mathbb{F}_q$.
Let $r(q)$ denote the reminder after division of $q$ by 4.
In \cite{BEHW96}, for any odd prime power $q$, a maximal clique in $P(q^2)$ of size $\frac{q+r(q)}{2}$ was constructed (note that this construction corresponds to the case $m=1$ in the proof of \cite[Theorem 5.2]{HS91}).
In \cite{GKSV18}, another maximal clique in $P(q^2)$ of the same size was constructed.
This clique was shown to have remarkable connection with eigenfunctions of $P(q^2)$ that have minimum cardinality of support $q+1$.
For more information on the problem of finding the minimum cardinality of support of eigenfunctions of graphs (MS-problem) and the problem of characterisation of the optimal eigenfunctions, see the survey \cite{SV21}.
The following problem belongs to this context.
    \begin{problem}\label{P2}
    Given an odd prime power $q$, what are eigenfunctions of $P(q^2)$ that have minimum cardinality of support?
    \end{problem}

The maximum cliques in a Paley graph $P(q^2)$ have size $q$, and the size $\frac{q+r(q)}{2}$ is second largest known for maximal cliques.
Thus the following problem is of interest.
    \begin{problem}\label{P3}
    Given an odd prime power $q$, are there maximal cliques in $P(q^2)$ whose size is greater than $\frac{q+r(q)}{2}$ and less than $q$?
    \end{problem}

It was found with use of computer that the cliques from \cite{BEHW96} and \cite{GKSV18} are the only maximal cliques of size $\frac{q+r(q)}{2}$ under the action of the automorphism group whenever $25 \leqslant q \leqslant 83$ holds.
On the other hand, there are extra cliques of the same size for $9 \leqslant q \leqslant 23$.

It follows from \cite[Theorem 3]{KMP16} that the subgraph of $P(q^2)$ induced by the vertices of the support of an optimal $\frac{-1-q}{2}$-eigenfunction is a complete bipartite graph.
Its parts induce cliques of size $\frac{q+1}{2}$ in the complementary graph, which is isomorphic to $P(q^2)$.
Thus a solution of the following problem would be an important step towards the solution of Problem \ref{P2}.
    \begin{problem}\label{probl}
    Given an odd prime power $q$, what are maximal cliques in $P(q^2)$ of size $\frac{q+r(q)}{2}$?
    \end{problem}

This paper is organised as follows.
In Section \ref{Preliminaries}, we list some useful notation and results.
In Section \ref{Constructions}, we reinterpret the constructions of maximal cliques and maximal independent sets of size $\frac{q+r(q)}{2}$ in $P(q^2)$ that were proposed in \cite{BEHW96, GKSV18}.
In Section \ref{Correspondence}, we introduce two linear fractional mappings, which establish correspondences between the sets from the constructions.
The results of the paper can be considered as contribution to the solution of Problem \ref{probl}.

\section{Preliminaries}\label{Preliminaries}

Let $q$ be an odd prime power.
In this section, we list some useful notation and results related to the affine plane $\AG(2,q)$, finite fields of square order, and automorphisms of $P(q^2)$.

\subsection{Affine plane $\AG(2,q)$}
Denote by $\AG(2,q)$ the point-line incidence structure, whose points are the vectors of the $2$-dimensional vector space $V(2,q)$ over $\mathbb{F}_q$, and the lines are the additive shifts of $1$-dimensional subspaces of $V(2,q)$. It is well known that $\AG(2,q)$ satisfies the axioms of a finite affine plane of order $q$.
In particular, each line contains $q$ points, and there exist $q+1$ lines through a point.
An \emph{oval} in the affine plane $\AG(2,q)$ is a set of $q+1$ points such that no three are on a line.
A line meeting an oval in one point (resp. in two points) is called \emph{tangent} (resp. \emph{secant}).
For any point of an oval, there exists a unique tangent at this point and $q$ secants.
By Qvist's theorem (see, for example, \cite[p.~147]{D68}),
given an oval in a projective plane (as well as an affine plane) of odd order and a point that does not belong to the oval, there are either $0$ or $2$ tangents to the oval through this point.

\subsection{Finite fields of square order}
Let $d$ be a non-square in $\mathbb{F}_{q}^*$. The elements of $\mathbb{F}_{q^2}$
can be considered as
$\{x+y\alpha:\,x,y \in \mathbb{F}_q\}$, where $\alpha$ is a root of the polynomial
$t^2 - d$.
Since $\mathbb{F}_{q^2}$ is a $2$-dimensional vector space over $\mathbb{F}_q$, we can assume that the points of $\AG(2,q)$ are the elements of $\mathbb{F}_{q^2}$, and a line $l$ is presented by the elements $\{x_1+y_1\alpha + c(x_2+y_2\alpha)\}$, where $x_1+y_1\alpha \in \mathbb{F}_{q^2}$, $x_2+y_2\alpha \in \mathbb{F}_{q^2}^*$ are fixed and $c$ runs over $\mathbb{F}_q$.
The element $x_2+y_2\alpha$ is called the \emph{slope} of the line $l$.
A line $l$ is called \emph{quadratic} (resp. \emph{non-quadratic}) if its slope is a square (resp. a non-square) in $\mathbb{F}_{q^2}^*$.
Note that the slope is defined up to multiplication by a constant $c \in \mathbb{F}_{q}^*$.
Let $\beta$ be a primitive element of the finite field $\mathbb{F}_{q^2}.$
Since $\mathbb{F}_{q}^* = \langle\beta^{q+1}\rangle$, the elements of $\mathbb{F}_{q}^*$ are squares in $\mathbb{F}_{q^2}^*$,
and the difference between any two points of a quadratic (resp. non-quadratic) line is a square (resp. a non-square) in $\mathbb{F}_{q^2}^*$.
In this setting, two distinct vertices are adjacent in $P(q^2)$ if and only if the line through these points is quadratic.

    \begin{proposition}\label{LinesThroughAPoint}
    For any point of $\AG(2,q)$, there exist exactly $\frac{q+1}{2}$ quadratic and $\frac{q+1}{2}$ non-quadratic lines through this point.
    \end{proposition}

\begin{proof}
Without loss of generality, consider the pencil of lines through the point $0$.
The vertex $0$ has $\frac{q^2-1}{2}$ neighbours in $P(q^2)$.
Each of them defines a quadratic line through $0$ in $\AG(2,q)$, which contains $q-1$ squares from $\mathbb{F}_{q^2}^*$.
So, there are $\frac{q^2-1}{2} \cdot \frac{1}{q-1}$ different quadratic lines in the pencil, that is,
$\frac{q+1}{2}$ lines.
Therefore,
the rest $\frac{q+1}{2}$ lines are non-quadratic. $\square$
\end{proof}

\bigskip
For any $\gamma = x+y\alpha\in \mathbb{F}_{q^2}^*$, define the \emph{norm} mapping $N:\,\mathbb{F}^*_{q^2} \mapsto \mathbb{F}^*_{q}$ by
$$N(\gamma) := \gamma^{q+1} = \gamma\gamma^{q} = (x+y\alpha)(x-y\alpha) = x^2 - y^2d.$$
This is a homomorphism with $\image(N) = \mathbb{F}^*_q$ and
$$\Ker(N) = \big\{\, x+y\alpha\in \mathbb{F}_{q^2}^*:~x^2 - y^2d = 1 \,\big\}.$$
The first isomorphism theorem for groups implies that $\Ker(N)$ is a subgroup in $\mathbb{F}_{q^2}^*$ of order $q+1$.
In addition, it is defined by a quadratic equation, so its elements form an oval in $\AG(2,q)$.

\bigskip
Let us take a look at some properties of squares in finite fields.

    \begin{proposition}\label{sq}
    The following statements hold.
    \begin{itemize}
    \item[\rm (1)] The element $-1$ is a square in $\mathbb{F}_q^*$ if and only if $q \equiv 1(4)$.
    \item[\rm (2)] For any non-square $n$ in $\mathbb{F}_q^*$, the element $-n$ is a square in $\mathbb{F}_q^*$ if and only if $q \equiv 3(4)$.
    \end{itemize}
    \end{proposition}

\begin{proof}
Let $\delta$ be a primitive element of $\mathbb{F}_{q}^*$.
Then
$\delta^{\frac{q-1}{2}} = -1$.
The left side is a square if and only if $q \equiv 1(4)$.
It is easy to see that the product of a square and a non-square (resp. two non-squares) in $\mathbb{F}_{q}^*$ is a non-square (resp. a square).
So, $-1 \cdot n$ is a square if and only if $q \equiv 3(4)$.
 $\square$
\end{proof}

    \begin{proposition}\label{sqq}
    An element $\gamma \in \mathbb{F}_{q^2}^*$ is a square if and only if $N(\gamma)$
    is a square
    in $\mathbb{F}_q^*$.
    \end{proposition}

\begin{proof}
Let $\gamma = (\gamma_1)^2$ for some $\gamma_1 \in \mathbb{F}_{q^2}^*$. Since $N$ is a homomorphism,
\begin{center}
$N(\gamma) = N \left( (\gamma_1)^2 \right) = \left( N(\gamma_1) \right)^2$.
\end{center}
It is well known that the even powers of a primitive element in a field of odd characteristic exhaust all non-zero squares.
So, there are $\frac{q^2-1}{2}$ squares in $\mathbb{F}_{q^2}^*$ and $\frac{q-1}{2}$ squares in $\mathbb{F}_{q}^*$.
Since $|\Ker(N)| = q+1$, all squares from $\mathbb{F}_{q^2}^*$ are mapped by $N$ to $\frac{q^2-1}{2} \cdot \frac{1}{q+1}$ elements from $\mathbb{F}_{q}^*$, that is,
$\frac{q-1}{2}$ elements.
As proved above, each of these elements is a square in $\mathbb{F}_{q}^*$.
Thus any preimage of a square from $\mathbb{F}_{q}^*$ is a square in $\mathbb{F}_{q^2}^*$.
$\square$
\end{proof}

    \begin{proposition}\label{sqqq}
    The element $\alpha$ is a square in $\mathbb{F}_{q^2}^*$ if and only if $q \equiv 3(4)$.
    \end{proposition}

\begin{proof}
It follows from Propositions \ref{sq}, \ref{sqq} and the equality $N(\alpha) = -d$. $\square$
\end{proof}

\bigskip
There are two special lines through $0$ in $\AG(2,q)$, for which we can decide in general if they are quadratic.

    \begin{proposition}\label{TwoSpecialLines}
    The following statements hold.
    \begin{itemize}
    \item[\rm (1)] The line $\{c:\,c \in \mathbb{F}_q\}$ is quadratic.
    \item[\rm (2)] The line $\{c\alpha:\,c \in \mathbb{F}_q\}$ is quadratic if and only if $q \equiv 3(4)$.
    \end{itemize}
    \end{proposition}

\begin{proof}
The line $\{c:\,c \in \mathbb{F}_q\}$ has the slope $1$, which is obviously a square.
The line $\{c\alpha:\,c \in \mathbb{F}_q\}$ has the slope $\alpha$, which is a square if and only if $q \equiv 3(4)$ in view of Proposition \ref{sqqq}.
$\square$
\end{proof}

\subsection{Automorphisms of $P(q^2)$}

The automorphism group of $P(q^2)$ is well studied and described below.

    \begin{proposition}[{\cite[Theorem 9.1]{J20}}]\label{G1}
    The automorphism group of $P(q^2)$ acts arc-transitively, and the equality
    $$\Aut \left( P(q^2) \right) = \big\{\, \gamma \mapsto a \gamma^\varepsilon+b:~a \in S,~b \in \mathbb{F}_{q^2},~\varepsilon \in {\rm Gal}(\mathbb{F}_{q^2}) \,\big\}$$
    holds, where $S$ is the set of square elements in $\mathbb{F}_{q^2}^*$.
    \end{proposition}

    \begin{corollary}\label{G1_lines}
    The group $\Aut \left( P(q^2) \right)$ preserves the set of quadratic lines and the set of non-quadratic lines. Moreover, it acts transitively on each of these sets.
    \end{corollary}

\begin{proof}
Every automorphism of $P(q^2)$ is a composition of a field automorphism, multiplication by a non-zero square from $\mathbb{F}_{q^2}^*$, and an additive shift by an element from $\mathbb{F}_{q^2}$.
It is easy to see that each of these operations preserves the sets of quadratic and non-quadratic lines, so the whole automorphism group does this as well.

Since the graph $P(q^2)$ is self-complementary, the group $\Aut \left( P(q^2) \right)$ acts arc-transitively on the complement of $P(q^2)$.
Recall that any two distinct adjacent (resp. non-adjacent) vertices of $P(q^2)$ uniquely determine a quadratic (resp. non-quadratic) line.
Thus, $\Aut \left( P(q^2) \right)$ acts transitively on the set of quadratic (resp. non-quadratic) lines. $\square$
\end{proof}

    \begin{corollary}\label{G1_points}
    There is a subgroup in $\Aut \left( P(q^2) \right)$ that stabilises the quadratic line $\mathbb{F}_q$ and acts faithfully on the set of points that do not belong to $\mathbb{F}_q$.
    \end{corollary}

\begin{proof}
This subgroup consists of the mappings $\gamma \mapsto a\gamma+b$,
where $a \in \mathbb{F}_q^*$ and $b \in \mathbb{F}_q$.
$\square$
\end{proof}

\section{Constructions}\label{Constructions}

In this section, we reinterpret the constructions of maximal cliques and maximal independent sets of size $\frac{q+r(q)}{2}$ in $P(q^2)$ that were proposed in \cite{BEHW96, GKSV18}.

\subsection{{\rm (}$\mathbb{F}_q,\alpha{\rm )}$-construction}

The first construction depends on the choice of a line and a point that is not on the line.
By Corollary \ref{G1_lines}, we may choose any line.
In view of Corollary \ref{G1_points}, given a line, we may choose any point that is not on the line.

Let us consider the quadratic line $\mathbb{F}_q$, the point $\alpha \notin \mathbb{F}_q$, and the pencil of quadratic lines through $\alpha$.
Each line of the pencil, except $\{c+\alpha:\,c \in \mathbb{F}_q\}$, intersects $\mathbb{F}_q$ at a point.
Denote by $c_1, \ldots, c_{\frac{q-1}{2}}$ all the intersection points.
So, the set $\{\alpha, c_1, \ldots, c_{\frac{q-1}{2}}\}$ is a clique of size $\frac{q+1}{2}$ in $P(q^2)$.
The points $-\alpha$ and $\alpha$ have the same neighbours in $\mathbb{F}_q$ due to Proposition \ref{sqq}.
Therefore, the set $\{-\alpha, c_1, \ldots, c_{\frac{q-1}{2}}\}$ is also a clique.
The first and the second cliques are equivalent under the field automorphism
$\gamma \mapsto \gamma^q$ (an analogue of the complex conjugation),
which acts as $\alpha \mapsto -\alpha$ and $c_k \mapsto c_k$.
By Proposition \ref{TwoSpecialLines}(2),
for $q \equiv 3(4)$, the line $\{c\alpha:\,c \in \mathbb{F}_q\}$, which contains $-\alpha$ and $\alpha$, is quadratic.
Hence the two cliques are combined into one.

    \begin{construction}\label{C1}
    If $q \equiv 1(4)$, then
    \begin{center}
    $\{\alpha, c_1, \ldots, c_{\frac{q-1}{2}}\}$
    is a maximal clique
    \end{center}
    of size $\frac{q+1}{2}$ in $P(q^2)$; if $q \equiv 3(4)$, then
    \begin{center}
    $\{\pm\alpha, c_1, \ldots, c_{\frac{q-1}{2}}\}$ is a maximal clique
    \end{center}
     of size $\frac{q+3}{2}$ in $P(q^2)$.
    \end{construction}

\subsection{{\rm(}$\alpha\mathbb{F}_q,1 {\rm)}$-construction}

The second construction is similar to Construction \ref{C1}.
It uses the line $\alpha \mathbb{F}_q$ and the point $1 \notin \alpha \mathbb{F}_q$.
For $q\equiv 1(4)$ (resp. $q\equiv 3(4)$), consider the pencil of non-quadratic (resp. quadratic) lines through $1$.
They all, except $\{1+c\alpha:\,c \in \mathbb{F}_q\}$, intersect $\alpha\mathbb{F}_q$.
Denote the intersection points by $c_1\alpha, \ldots, c_{\frac{q-1}{2}}\alpha$.
By Proposition \ref{sqq}, the elements $-1$ and $1$ have the same neighbours in $\alpha\mathbb{F}_q$.
Thus $\{1, c_1\alpha, \ldots, c_{\frac{q-1}{2}}\alpha\}$ and $\{-1, c_1\alpha, \ldots, c_{\frac{q-1}{2}}\alpha\}$ are independent sets (resp. cliques) of size $\frac{q+1}{2}$ in $P(q^2)$.
They are equivalent under the automorphism $\gamma \mapsto -\gamma^q$ (resp. they are equivalent under the automorphism $\gamma \mapsto -\gamma^q$ and combined into one due to Proposition \ref{TwoSpecialLines}(1)).

    \begin{construction}\label{C2}
    If $q \equiv 1(4)$, then
    \begin{center}
    $\{1, c_1\alpha, \ldots, c_{\frac{q-1}{2}}\alpha\}$
    is a maximal independent set
    \end{center}
    of size $\frac{q+1}{2}$ in $P(q^2)$; if $q \equiv 3(4)$, then
    \begin{center}
    $\{\pm1, c_1\alpha, \ldots, c_{\frac{q-1}{2}}\alpha\}$ is a maximal clique
    \end{center}
     of size $\frac{q+3}{2}$ in $P(q^2)$.
    \end{construction}

The next Lemma shows that the sets from Constructions \ref{C2} and \ref{C1} are equivalent under the mapping $\gamma \mapsto \alpha \gamma$.

    \begin{lemma}\label{multiplicationByAlpha}
    Let $C$ be a maximal clique (resp. a maximal independent set) in $P(q^2)$. If $q \equiv 1(4)$, then
    \begin{center}
    $\alpha C$ is a maximal independent set (resp. a maximal clique) in $P(q^2)$;
    \end{center}
    if $q \equiv 3(4)$, then
    \begin{center}
    $\alpha C$ is a maximal clique (resp. a maximal independent set) in $P(q^2)$.
    \end{center}
    \end{lemma}

\begin{proof}
It follows from Proposition \ref{sqqq} and the fact that $P(q^2)$ is self-complementary.
$\square$
\end{proof}

\subsection{$Q$-construction}

The third construction was proposed in \cite{GKSV18}.
Let $\beta$ be a primitive element in $\mathbb{F}_{q^2}$ and $\omega = \beta^{q-1}$.
Then $\omega$ is a square in $\mathbb{F}_{q^2}^*$, and $\langle\omega\rangle$ is a subgroup of order $q+1$ in $\mathbb{F}_{q^2}^*$.
Denote it by $Q$.
On the other hand, $\Ker(N)$ is also a subgroup of order $q+1$ in $\mathbb{F}_{q^2}^*$.
Since the subgroups of a finite cyclic group are uniquely determined by the divisors of the group order,
$Q$ and  $\Ker(N)$ coincide.
Thus the points of $Q$ form an oval in $\AG(2,q)$.
Put
\begin{center}
$Q_0 := \langle\omega^2\rangle$ \quad and \quad $Q_1 := \omega  \langle\omega^2\rangle$.
\end{center}
Obviously, $Q = Q_0 \cup Q_1$ and $Q_1 = \omega Q_0$.

    \begin{construction}[{\cite[Theorem 1]{GKSV18}}]\label{C3}
    If $q \equiv 1 (4)$, then
    \begin{center}
    $Q_0$ and $Q_1$ are maximal independent sets
    \end{center}
    of size $\frac{q+1}{2}$ in $P(q^2)$;
    if $q \equiv 3 (4)$, then
    \begin{center}
    $Q_0 \cup \{0\}$ and $Q_1 \cup \{0\}$ are maximal cliques
    \end{center}
    of size $\frac{q+3}{2}$ in $P(q^2)$.
    \end{construction}

\subsection{$\alpha Q$-construction}

The fourth construction is a corollary of Construction \ref{C3} and Lemma \ref{multiplicationByAlpha}.

    \begin{construction}\label{C4}
    If $q \equiv 1 (4)$, then
    \begin{center}
    $\alpha Q_0$ and $\alpha Q_1$ are maximal cliques
    \end{center}
    of size $\frac{q+1}{2}$ in $P(q^2)$;
    if $q \equiv 3 (4)$, then
    \begin{center}
    $\alpha Q_0 \cup \{0\}$ and $\alpha Q_1 \cup \{0\}$ are maximal cliques
    \end{center}
    of size $\frac{q+3}{2}$ in $P(q^2)$.
    \end{construction}

\section{Correspondence between constructions}\label{Correspondence}

In this section, we introduce two linear fractional mappings, which establish correspondences between the sets from Constructions \ref{C3} and \ref{C2}, and between the sets from Constructions \ref{C4} and \ref{C1}.

\subsection{$\varphi$-correspondence between $Q$- and {\rm(}$\alpha\mathbb{F}_q,1 {\rm)}$-constructions}
\label{mappingPhi}

Let us define
the first mapping $\varphi:\mathbb{F}_{q^2} \mapsto \mathbb{F}_{q^2}$ by the rule
$$
\varphi(\gamma):=
\begin{cases}
\frac{\gamma+1}{\gamma-1} & \text{if~~}  \gamma \not= 1,\\
~~1 & \text{if~~} \gamma=1.
\end{cases}
$$
Take a look at its properties.

\begin{lemma}\label{phiIsBijAndInv}
The mapping $\varphi$ is a bijection and an involution.
\end{lemma}
\begin{proof}
First, we show that $\varphi$ is a bijection.
Since we are dealing with finite fields,
it suffices to prove that $\varphi$ is an injection.
Let $\gamma_1,\gamma_2 \in \mathbb{F}_{q^2} \setminus \{1\}$.
Suppose $\varphi(\gamma_1) = \varphi(\gamma_2)$ or, equivalently,
$$
\frac{\gamma_1+1}{\gamma_1-1} = \frac{\gamma_2+1}{\gamma_2-1}.
$$
After standard manipulations, this equation becomes $\gamma_1 = \gamma_2$.
So, if $\gamma_1 \neq \gamma_2$ then $\varphi(\gamma_1) \neq \varphi(\gamma_2)$.
In addition, if $\varphi(\gamma_3)=1$ for some $\gamma_3 \in \mathbb{F}_{q^2}$, then the only possible value for $\gamma_3$ is $1$ (otherwise, we have a contradiction).
Thus the image $\varphi(\mathbb{F}_{q^2} \setminus \{1\})$ has
$q^2-1$ elements and does not contain $1$,
which means that $\varphi$ is an injection and, consequently, a bijection.

Second, we prove that $\varphi$ is an involution.
By definition, we need to check that $\varphi^2$ is the identity mapping.
Let $\gamma \in \mathbb{F}_{q^2} \setminus \{1\}$.
Then
$$
\varphi^2(\gamma) = \frac{\varphi(\gamma)+1}{\varphi(\gamma)-1} =
\frac{\frac{\gamma+1}{\gamma-1}+1}{\frac{\gamma+1}{\gamma-1}-1}.
$$
It is easy to see that the last expression equals $\gamma$. $\square$
\end{proof}

    \begin{proposition}\label{Q0Image1}
    Let $\gamma$ be an element from $Q \setminus \{1\}$, where $\gamma = x + y\alpha$ for some $x,y \in \mathbb{F}_{q}$. Then the following formula holds
    \begin{center}
        $\displaystyle \varphi(\gamma) = \frac{y}{x-1}\alpha$.
    \end{center}
    \end{proposition}
    \begin{proof}
        Substitute the expression for $\gamma$ into the definition of $\varphi$:
        $$
        \varphi(\gamma) = \frac{x+y\alpha+1}{x+y\alpha-1} = \frac{(x+1+y\alpha)(x-1-y\alpha)}{(x-1+y\alpha)(x-1-y\alpha)} = \frac{x^2-(1+y\alpha)^2}{(x-1)^2-y^2d}.
        $$
        Expand the brackets and use the equation for $Q$, that is, $x^2 - y^2d = 1$:
        $$
        \varphi(\gamma) = \frac{x^2-1-2y\alpha -y^2d}{x^2-2x+1-y^2d} = \frac{y}{x-1}\alpha.
        ~ \square
        $$
    \end{proof}

    \begin{corollary}
        The set $Q\setminus \{1\}$ is mapped to $\alpha\mathbb{F}_{q}$ by $\varphi$ bijectively.
    \end{corollary}

    \begin{lemma}\label{Q0Image2}
    Let $\gamma$ be an element from $Q \setminus \{1, -1\}$, where $\gamma = x + y\alpha$ for some $x,y \in \mathbb{F}_{q}$.
    Then the following formula holds
    \begin{center}
        $\displaystyle \varphi(\gamma^2) = \frac{x}{yd}\alpha$.
    \end{center}
    \end{lemma}
    \begin{proof}
        Introduce temporary variables
        $$\gamma^2 = (x+y\alpha)^2 = \underbrace{x^2+y^2d}_{x_1}+\underbrace{2xy}_{y_1}\alpha,
        $$
        and apply Proposition \ref{Q0Image1}:
        $$
        \varphi \left( \gamma^2 \right) = \varphi(x_1+y_1\alpha) = \frac{y_1}{x_1-1}\alpha.
        $$
        Return to the original notation and take into account the equation for $Q$:
        $$
        \varphi \left( \gamma^2 \right) = \frac{2xy}{x^2+y^2d-1}\alpha = \frac{2xy}{2y^2d}\alpha = \frac{x}{yd}\alpha.
        ~ \square
        $$
    \end{proof}

The first of two main results of the paper is formulated below.

    \begin{theorem} \label{3corresponds2}
    Let $Q_0$ be the set from Construction \ref{C3}.
    If $q \equiv 1(4)$, then $\varphi(Q_0)$ coincides with the maximal independent set
    from Construction \ref{C2}, that is,
    $$
    \varphi(Q_0) = \{1, c_1\alpha, \ldots, c_{\frac{q-1}{2}}\alpha\}.
    $$
    If $q \equiv 3(4)$, then $\varphi(Q_0 \cup \{0\})$ coincides with the maximal clique from Construction \ref{C2}, that is,
    $$
    \varphi(Q_0 \cup \{0\}) = \{\pm 1, c_1\alpha, \ldots, c_{\frac{q-1}{2}}\alpha\}.
    $$
    \end{theorem}

\begin{proof}
Recall that $-1$ and $1$ have the same neighbours in $\alpha\mathbb{F}_q$.
So, for $q\equiv 1(4)$ (resp. $q\equiv 3(4)$), it suffices to show that $1$ has no neighbours in $\varphi(Q_0 \setminus \{1\})$ (resp. $1$ is adjacent to all elements of $\varphi(Q_0 \cup \{0\} \setminus \{1\})$). Note that
$$
\varphi(-1) = 0, \quad \varphi(0) = -1, \quad \varphi(1) = 1.
$$
Take an arbitrary element $\gamma \in Q_0 \setminus \{1\}$.
By definition, $Q_{0}$ consists of all squares from $Q$.
So, $\gamma = (\gamma_1)^2$ for some $\gamma_1 \in Q \setminus \{1, -1\}$.
Let $\gamma_1 = x+y\alpha$.
Then $\varphi(\gamma) = \frac{x}{yd}\alpha$ holds due to Lemma \ref{Q0Image2}.

Now check when the elements $1$ and $\frac{x}{yd}\alpha$ are adjacent in $P(q^2)$.
To do this, calculate the norm of their difference:
$$
N\left(1-\frac{x}{yd}\alpha\right) = 1 - \frac{x^2d}{y^2d^2} = 1 - \frac{x^2}{y^2d} = \frac{y^2d - x^2}{y^2d} = -\frac{1}{y^2d} = -1 \cdot \frac{1}{y^2} \cdot \frac{1}{d}.
$$
In the last product, $\frac{1}{d}$ is a non-square; $\frac{1}{y^2}$ is a square; by Proposition \ref{sq}(1), $-1$ is a square if and only if $q\equiv 1(4)$.
Hence $-\frac{1}{y^2d}$ is a non-square if and only if $q\equiv 1(4)$.
Thus for $q\equiv 1(4)$ (resp. $q\equiv 3(4)$),
the elements $1$ and $\frac{x}{yd}\alpha$ are non-adjacent (resp. adjacent). $\square$
\end{proof}

    \begin{corollary}
    If $q \equiv 1(4)$, then $\varphi(Q_1 \cup \{1\})$ induces a complete bipartite graph with parts $\varphi(Q_1)$ and $\{1\}$.
    If $q \equiv 3(4)$, then $\varphi(Q_1 \cup \{0,1\})$ induces a disjoint union of the clique $\varphi(Q_1)$ and the edge $\{1,-1\}$.
    \end{corollary}

\subsection{$\psi$-correspondence between $\alpha Q$- and
{\rm (}$\mathbb{F}_q,\alpha{\rm )}$-constructions}

Define the second mapping $\psi:\mathbb{F}_{q^2} \mapsto \mathbb{F}_{q^2}$ as
$$
\psi(\gamma):= \alpha \varphi(\alpha^{-1} \gamma) =
\begin{cases}
\frac{\alpha\gamma+d}{\gamma-\alpha} & \text{if~~} \gamma \not= \alpha,\\
~~~\alpha & \text{if~~} \gamma=\alpha.
\end{cases}
$$
Its properties are direct corollaries of the results from Section \ref{mappingPhi}.

    \begin{proposition}\label{allAboutAlphaQ}
    The following statements hold.
    \begin{itemize}
    \item[\rm (1)] The mapping $\psi$ is a bijection and an involution.
    \item[\rm (2)] Let $\gamma$ be an element from $Q \setminus \{1\}$, where $\gamma = x + y\alpha$ for some $x,y \in \mathbb{F}_{q}$.
    Then
    $$
    \displaystyle \psi(\alpha\gamma) = \frac{yd}{x-1}.
    $$
    \item[\rm (3)] The set $\alpha Q\setminus \{\alpha\}$ is mapped to $\mathbb{F}_{q}$ by $\psi$ bijectively.
    \item[\rm (4)]
    Let $\gamma$ be an element from $Q \setminus \{1, -1\}$, where $\gamma = x + y\alpha$ for some $x,y \in \mathbb{F}_{q}$.
    Then
    $$
    \displaystyle \psi(\alpha\gamma^2) = \frac{x}{y}.
    $$
    \end{itemize}
    \end{proposition}

We conclude this section with the second main result of the paper.
It follows from Theorem \ref{3corresponds2} and Lemma \ref{multiplicationByAlpha}.

    \begin{theorem}
    Let $\alpha Q_0$ be the set from Construction \ref{C4}.
    If $q \equiv 1(4)$, then $\psi(\alpha Q_0)$ coincides with the maximal clique
    from Construction \ref{C1}, that is,
    $$
    \psi(\alpha Q_0) = \{\alpha, c_1, \ldots, c_{\frac{q-1}{2}}\}.
    $$
    If $q \equiv 3(4)$, then $\psi(\alpha Q_0 \cup \{0\})$ coincides with the maximal clique from Construction \ref{C1}, that is,
    $$
    \psi(\alpha Q_0 \cup \{0\}) = \{\pm \alpha, c_1, \ldots, c_{\frac{q-1}{2}}\}.
    $$
    \end{theorem}

    \begin{corollary}
    If $q \equiv 1(4)$, then $\psi(\alpha Q_1 \cup \{\alpha\})$ is
    a disjoint union of the clique
     $\psi(\alpha Q_1)$ and the vertex $\alpha$.
    If $q \equiv 3(4)$, then $\psi(\alpha Q_1 \cup \{0,\alpha\})$ is a disjoint union of the clique $\psi(\alpha Q_1)$ and the edge $\{\alpha,-\alpha\}$.
    \end{corollary}

\section*{Acknowledgments}
Alexander Masley is supported by
a grant from the United States-Israel Binational Science Foundation (BSF), Jerusalem, Israel,
and by The Center for Absorption in Science, Ministry of Immigrant  Absorption, State of Israel. Leonid Shalaginov is supported by RFBR according to the research project
20-51-53023.
The authors are grateful to the anonymous referee for careful reading and helpful comments, which significantly improved the paper.
The authors are also grateful to Chi Hoi Yip for pointing out the reference \cite[Theorem 5.2]{HS91}.
Special thanks go to Rhys Evans for useful remarks concerning the paper at its final stage.

\newpage
\section*{Appendix}

\begin{figure}[h]
\begin{center}
\begin{minipage}[h]{0.45\linewidth}
\includegraphics[width=1\linewidth]{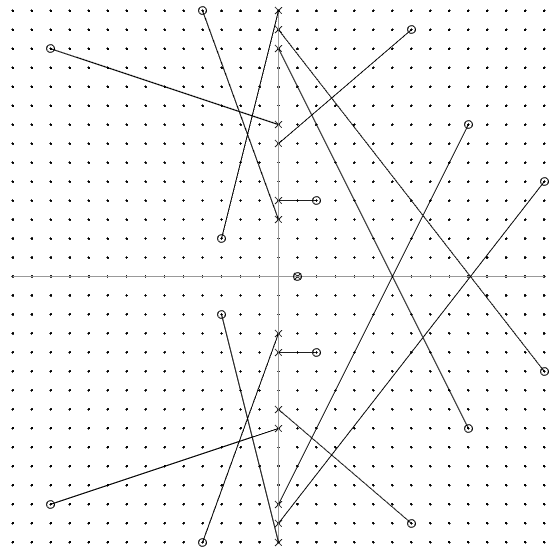}
\caption{$q=29$, $Q_0$, $\varphi(Q_0)$}
\end{minipage}
\quad
\begin{minipage}[h]{0.45\linewidth}
\includegraphics[width=1\linewidth]{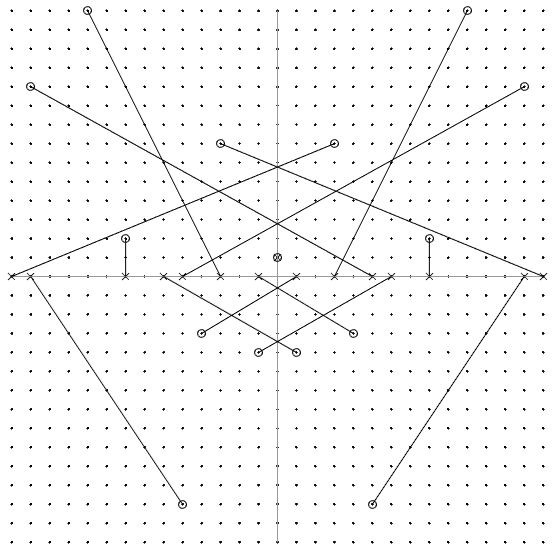}
\caption{$q=29$, $\alpha Q_0$, $\psi(\alpha Q_0)$}
\end{minipage}
\end{center}
\end{figure}

Let $q=29$ and $d=2$.
Then the sets from Constructions $1$ -- $4$ are
\begin{center}
$\left\{ \alpha,~\pm1,~\pm3,~\pm5,~\pm6,~\pm8,~\pm13,~\pm14 \right\},$ \\[4.5pt]
$\left\{ 1, \pm3\alpha,~\pm4\alpha,~\pm7\alpha,~\pm8\alpha,~\pm12\alpha,~\pm13\alpha,~\pm14\alpha \right\},$ \\[4.5pt]
$Q_0 = \left\{ -12 \pm 12\alpha,~-4 \pm 14\alpha,~-3 \pm 2\alpha,~1, \right.$ \\[3pt]
$\left.2 \pm 4\alpha,~7 \pm 13\alpha,~10 \pm 8\alpha,~14 \pm 5\alpha \right\},$ \\[4.5pt]
$\alpha Q_0 = \left\{\pm 5-12\alpha,~\pm 1-4\alpha,~\pm 4-3\alpha,~\alpha, \right.$ \\[3pt]
$\left. \pm 8 + 2\alpha,~\pm 3 + 7\alpha,~\pm13+10\alpha,~\pm 10+14\alpha \right\}.$
\end{center}
The mappings
$$\varphi: Q_0 \mapsto \left\{ 1, \ldots, \pm 14 \alpha \right\}
\quad \text{and} \quad
\psi: \alpha Q_0 \mapsto \left\{ \alpha, \ldots, \pm 14 \right\}$$
act as follows (see also Fig. 1 and Fig. 2):

\bigskip
\noindent
\begin{minipage}{0.5\textwidth}
\small
\begin{center}
    \begin{tabular}{||c|c||}
    \hline
    $Q_0$ & $\varphi(Q_0)$ \\
    \hline \hline
    $-3 + 2\alpha$ & $14\alpha$ \\
    \hline
    $14 - 5\alpha$ & $13\alpha$ \\
    \hline
    $10 - 8\alpha$ & $12\alpha$ \\
    \hline
    $-12 + 12\alpha$ & $8\alpha$ \\
    \hline
    $7 + 13\alpha$ & $7\alpha$ \\
    \hline
    $2 + 4\alpha$ & $4\alpha$ \\
    \hline
    $-4 + 14\alpha$ & $3\alpha$ \\
    \hline
    $1$ & $1$ \\
    \hline
    $-4 - 14\alpha$ & $-3\alpha$ \\
    \hline
    $2 - 4\alpha$ & $-4\alpha$ \\
    \hline
    $7 - 13\alpha$ & $-7\alpha$ \\
    \hline
    $-12 - 12\alpha$ & $-8\alpha$ \\
    \hline
    $10 + 8\alpha$ & $-12\alpha$ \\
    \hline
    $14 + 5\alpha$ & $-13\alpha$ \\
    \hline
    $-3 - 2\alpha$ & $-14\alpha$ \\
    \hline
    \end{tabular}
    \end{center}
\end{minipage}
\begin{minipage}{0.5\textwidth}
\small
    \begin{center}
    \begin{tabular}{||c|c||}
    \hline
    $\alpha Q_0$ & $\psi(\alpha Q_0)$ \\
    \hline \hline
    $-3 + 7\alpha$ & $14$ \\
    \hline
    $5-12\alpha$ & $13$ \\
    \hline
    $8 + 2\alpha$ & $8$ \\
    \hline
    $-1-4\alpha$ & $6$ \\
    \hline
    $-13+10\alpha$ & $5$ \\
    \hline
    $10+14\alpha$ & $3$ \\
    \hline
    $-4-3\alpha$ & $1$ \\
    \hline
    $\alpha$ & $\alpha$ \\
    \hline
    $4-3\alpha$ & $-1$ \\
    \hline
    $-10+14\alpha$ & $-3$ \\
    \hline
    $13+10\alpha$ & $-5$ \\
    \hline
    $1-4\alpha$ & $-6$ \\
    \hline
    $-8 + 2\alpha$ & $-8$ \\
    \hline
    $-5-12\alpha$ & $-13$ \\
    \hline
    $3 + 7\alpha$ & $-14$ \\
    \hline
    \end{tabular}
    \end{center}
\end{minipage}

\bigskip
\begin{figure}[h]
\begin{center}
\begin{minipage}[h]{0.45\linewidth}
\includegraphics[width=1\linewidth]{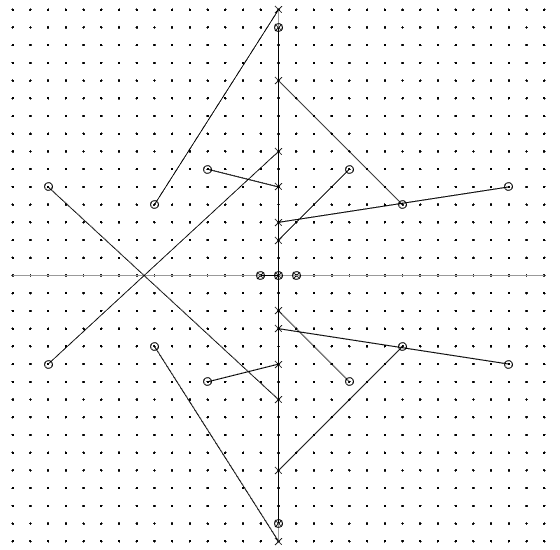}
\caption{$q=31$, $Q_0 \cup \{0\}$, $\varphi(Q_0 \cup \{0\})$}
\end{minipage}
\quad
\begin{minipage}[h]{0.45\linewidth}
\includegraphics[width=1\linewidth]{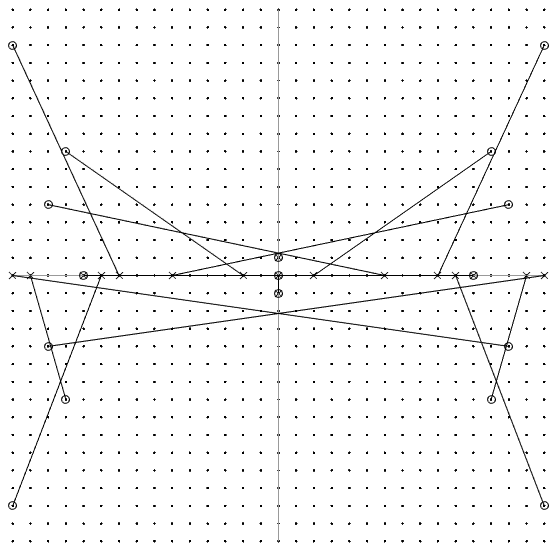}
\caption{$q=31$, $\alpha Q_0 \cup \{0\}$, $\psi(\alpha Q_0 \cup \{0\})$}
\end{minipage}
\end{center}
\end{figure}

\bigskip
Let $q=31$ and $d=3$.
Then the sets from Constructions $1$ -- $4$ are
\begin{center}
$\left\{
\pm\alpha,~0,~\pm2,~\pm6,~\pm9,~\pm10,~\pm11,~\pm14,~\pm15
\right\},$ \\[4.5pt]
$\left\{ 0,~\pm1,~\pm2\alpha,~\pm3\alpha,~\pm5\alpha,~\pm7\alpha,~\pm11\alpha,~\pm14\alpha,~\pm15\alpha \right\},$ \\[4.5pt]
$Q_0 \cup \{0\} = \left\{0,~\pm 1,~\pm4 \pm 6\alpha,~\pm 7 \pm 4\alpha,~\pm 13 \pm 5\alpha,~\pm14\alpha \right\},$ \\[4.5pt]
$\alpha Q_0 \cup \{0\} = \left\{0,~\pm11,~\pm\alpha,~\pm 13 \pm 4\alpha,~\pm 12 \pm 7\alpha,~ \pm 15 \pm 13\alpha
    \right\}.$
\end{center}

The mappings
$$\varphi: Q_0 \cup \{0\} \mapsto \left\{ 0, \ldots, \pm 15\alpha \right\} \quad \text{and} \quad \psi: \alpha Q_0 \cup \{0\} \mapsto \left\{ \pm \alpha, \ldots, \pm 15 \right\}$$
act as follows (see Fig. 3 and Fig. 4):

\bigskip
\noindent
\begin{minipage}{0.5\textwidth}
\small
    \begin{center}
    \begin{tabular}{||c|c||}
    \hline
    $Q_0 \cup \{0\}$ & $\varphi(Q_0 \cup \{0\})$ \\
    \hline \hline
    $-7 + 4\alpha$ & $15\alpha$ \\
    \hline
    $-14\alpha$ & $14\alpha$ \\
    \hline
    $7 + 4\alpha$ & $11\alpha$ \\
    \hline
    $-13 - 5\alpha$ & $7\alpha$ \\
    \hline
    $-4 + 6\alpha$ & $5\alpha$ \\
    \hline
    $13 + 5\alpha$ & $3\alpha$ \\
    \hline
    $4 + 6\alpha$ & $2\alpha$ \\
    \hline
    $1$ & $1$ \\
    \hline
    $-1$ & $0$ \\
    \hline
    $0$ & $-1$ \\
    \hline
    $4 - 6\alpha$ & $-2\alpha$ \\
    \hline
    $13 - 5\alpha$ & $-3\alpha$ \\
    \hline
    $-4 - 6\alpha$ & $- 5\alpha$ \\
    \hline
    $-13 + 5\alpha$ & $- 7\alpha$ \\
    \hline
    $7 - 4\alpha$ & $-11\alpha$ \\
    \hline
    $14\alpha$ & $-14\alpha$ \\
    \hline
    $-7 - 4\alpha$ & $-15\alpha$ \\
    \hline
    \end{tabular}
    \end{center}
\end{minipage}
\begin{minipage}{0.5\textwidth}
\small
    \begin{center}
    \begin{tabular}{||c|c||}
    \hline
    $\alpha Q_0 \cup \{0\}$ & $\psi(\alpha Q_0 \cup \{0\})$ \\
    \hline \hline
    $-13 - 4\alpha$ & $15$ \\
    \hline
    $12-7\alpha$ & $14$ \\
    \hline
    $-11$ & $11$ \\
    \hline
    $15 -13\alpha$ & $10$ \\
    \hline
    $15+13\alpha$ & $9$ \\
    \hline
    $-13 + 4\alpha$ & $6$ \\
    \hline
    $12+7\alpha$ & $2$  \\
    \hline
    $\alpha$ & $\alpha$ \\
    \hline
    $-\alpha$ & $0$ \\
    \hline
    $0$ & $-\alpha$ \\
    \hline
    $-12+7\alpha$ & $-2$  \\
    \hline
    $13 + 4\alpha$ & $-6$ \\
    \hline
    $-15+13\alpha$ & $-9$ \\
    \hline
    $-15 -13\alpha$ & $-10$ \\
    \hline
    $11$ & $-11$ \\
    \hline
    $-12-7\alpha$ & $-14$ \\
    \hline
    $13 - 4\alpha$ & $-15$ \\
    \hline
    \end{tabular}
    \end{center}
\end{minipage}

\end{document}